\documentclass[12pt,letter]{amsart} 

\usepackage{amsmath} 
\usepackage{amssymb}
\usepackage{graphicx}
\graphicspath{ {Thesis/} }
\usepackage[cmtip,all]{xy}
\usepackage{amsthm}
\usepackage{caption}
\usepackage{subcaption}
\usepackage[margin=1in]{geometry}
\usepackage{mathrsfs}
\usepackage[fleqn,tbtags]{mathtools}
\usepackage{parskip}
\usepackage{multicol}
\usepackage{float}

\newtheorem{theorem}{Theorem}
\newtheorem{corollary}{Corollary}[theorem]

\newtheorem{proposition}[theorem]{Proposition}

\theoremstyle{definition}
\newtheorem{definition}[theorem]{Definition}

\theoremstyle{remark}
\newtheorem*{remark}{Remark}

\numberwithin{equation}{section}

\newcommand{\h}{{\mathscr H }} %Use for Heegaard graph and vertices/edges of Heegaard graph
\newcommand{\hand}{{\mathcal H }}
\begin{document}

\title{A Seifert algorithm for integral homology spheres.}

\author[Alegria and Menasco]{Linda V. Alegria and William W. Menasco}
\address{Department of Mathematics, University at Buffalo}
\email{lindaale@buffalo.edu}

\address{Department of Mathematics, University at Buffalo}
\email{menasco@buffalo.edu}

\keywords{Heegaaard splitting, Heegaard diagrams, homology spheres, 3-Manifolds, Seifert's algorithm, spanning surface.}

\date{\today}

\begin{abstract}

    From classical knot theory we know that every knot in $S^3$ is the boundary of an oriented, embedded surface.  A standard demonstration of this fact achieved by elementary technique comes from taking a regular projection of any knot and employing Seifert's constructive algorithm.  In this note we give a natural generalization of Seifert's algorithm to any closed integral homology 3-sphere.  The starting point of our algorithm is presenting the handle structure of a Heegaard splitting of a given integral homology sphere as a planar diagram on the boundary of a $3$-ball.  (For a well known example of such a planar presentation, see the Poincar\'e homology sphere planar presentation in {\em Knots and Links} by D. Rolfsen \cite{Rolfsen}.)  An oriented link can then be represented by the regular projection of an oriented $k$-strand tangle.  From there we give a natural way to find a ``Seifert circle" and associated half-twisted bands.
 
\end{abstract}

\maketitle

\section{Introduction}
\label{section: intro}

In Herbert Seifert's seminal paper, ``\"Uber das Geschlecht von Knoten'' \cite{Seifert}, a simple constructive algorithm is described that starts with a regular projection of a given oriented knot (or link) on an $S^2$ in $S^3$ and yields an oriented surface having the given knot (or link) as its boundary---notably the surface has a visible presentation that ``overlays'' the knot projection.  From there Seifert gives us a {\em tour de force}, defining a matrix that describes the homology relations of his {\em Seifert surface}, then utilizes this {\em Seifert matrix} to produce an efficient method of computing the Alexander polynomial.  So constructed, he then demonstrates that half the degree of the Alexander polynomial is a lower bound for the genus of the Seifert surface. Finally, he investigates the structure of polynomials that can be Alexander polynomials, even producing examples of non-trivial knots that have trivial Alexander polynomial.

The goal of this manuscript is to give a natural generalization of Seifert's constructive algorithm to any closed integral homology $3$-sphere.  To see the naturality of this generalization it is helpful to be more descriptive of Seifert's elegant algorithm.

Seifert's algorithm starts with a planar presentation of a genus-$0$ Heegaard diagram of $S^3$.  That is, a $2$-sphere that gives a handle decomposition---one $3$-dimensional $0$-handle and one $3$-dimensional $3$-handle.  The oriented knot in question is projected onto this $2$-sphere by a regular projection.  Utilizing the orientation of the knot, the crossings of the projection are resolved and the resultant oriented {\em Seifert circles} bound a collection of disjoint discs in the $0$-handle of $S^3$.  Thus, we have $2$-dimensional $0$-handles which are the start of a spanning surface positioned in the $0$-handle of our $3$-manifold's handle decomposition.  We recover the resolved crossing of the knot projection by attaching to this collection of disjoint discs half-twisted bands at each crossing.  The resulting spanning surface should then be thought of as being in ``pseudo-normal position''.  That is, although the surface's $0$-handles are well positioned with respect to the handle decomposition of $S^3$, this decomposition lacks $3$-dimensional $1$-handles in which the $2$-dimensional $1$-handles of Seifert's construction---the half-twisted bands---might be positioned.  Instead, these half-twisted bands are positioned arbitrarily close to the shared boundary sphere of the $0$-handle and $3$-handle in the handle decomposition of $S^3$.

Seifert's algorithm can be thought of as a demonstration that $S^3$ {\em is} a homology $3$-sphere, having trivial first integral homology. Our constructive algorithm will start with a given oriented link in any integral homology sphere, $M$.  The construction is based upon representing $M$ via a Heegaard splitting or, equivalently, a handle decomposition having exactly one $0$-handle.  An oriented knot, $K \subset M$, is then in normal position in the $0$- and $1$-handles of the decomposition.  (Please see \S\ref{section: definitions} for formal definition of handle decomposition and normal position.)  With respect to this normal position, if $K$ goes through the $1$-handles $n$ times, then $K$ can be represented by a regular projection of an $n$-tangle on the boundary of the unique $0$-handle.  If we are considering oriented links then we must allow the $n$-tangle to possibly have $S^1$ components that are totally contained in the unique $0$-handle.

The algorithm will produce a spanning surface which is in pseudo-normal position with respect to the handle decomposition of $M$.  That is, the spanning surface will also have a handle decomposition where its $0$- and $2$-handles are properly embedded in the $0$- and $2$-handles of $M$, respectively.  Additionally, its $1$-handles come in two flavors: those properly embedded in the $1$-handles of $M$ and those that are half-twisted bands at crossings of the $n$-tangle projection in an arbitrarily close neighborhood of the boundary of the $0$-handle of $M$.

Our paper is organized as follows.  In \S\ref{section: definitions} we review the needed definitions of handle decompositions for $2$-manifolds with boundary and closed oriented $3$-manifolds.

In \S\ref{subsection: algorithm} we describe the spanning surface algorithm.  Since there is a needed homological calculation for determining how many times a spanning surface intersects a given $2$-handle of $,M$, we provide the reader with an example of knots in the Poincar\'e homology sphere in \S\ref{section: P-homology sphere}.  

The example discussed in \S\ref{section: P-homology sphere} is based upon a geometric calculation---taking $+1$ Dehn surgery on the trefoil and producing a Heegaard diagram---lifted from D. Rolfsen's seminal book, {\em Knots and Links} \cite{Rolfsen}.  We offer a method for generalizing this calculation for arbitrary framed links in \S\ref{section: Constructing Heegaard graph}.

We end the paper with an open question in section \S\ref{section: questions}.

\section{Definitions}
\label{section: definitions}

First, we specify some of our notation.  The cardinality of a set or collection, $X$, will be denoted by $|X|$.  The interior of a topological set, $X$, will be denoted by, $int(X)$. 
For an oriented, $n$-dimensional manifold $M^n$, $-M^n$ will denote the same manifold with orientation reversed. An orientation preserving and orientation reversing map will be denoted by $\phi^+$ and $\phi^-$, respectively.
$B^k$ will be notation for a closed ball of dimension $k$ and $c(B^k)$ will denote the {\em center point} of the $k$-ball. The sign function, which we will denote with $\sigma(x)$ is defined as follows: 
$\sigma(x) =
\begin{cases}
    + & \text{ if } |x| = x \\
    - & \text{ otherwise}
\end{cases}$

To facilitate a straightforward expository, we begin with the needed definitions specialized to dimensions $0$, $1$, $2$ and $3$.

\subsection{Handle Structure and Orientation}
An {\em $n$-dimensional $k$-handle} is a $n$-ball having product structure, $B^k \times B^{n-k}$.  The {\em core} of a $k$-handle is $C^k := B^k \times c(B^{n -k})$, and the {\em co-core} is $C^{k \ast} := c(B^k) \times B^{n -k}$.

There are then two natural projection maps, $\pi: B^k \times B^{n-k} \rightarrow C^k$, and $\pi^\ast : B^k \times B^{n-k} \rightarrow C^{k \ast}$.

For $n=2,3$ we begin with an oriented 0-handle. We then choose an orientation for the core, $C^1$, of the 1-handle. This induces an orientation on the co-core, $C^{1*}$, of the 1-handle (such that the algebraic intersection between $C^1$ and $C^{1*}$ is positive). Now, with $C^1$ oriented, we note that the boundary of $C^1$ consists of two points which we can appropriately denote by $c^+$ and $c^-$. Additionally, we will denote by $\mathbb{B}^{\pm} := c^{\pm} \times B^{n-k}$ the attaching region of the 1-handle. For the 2-handle, choose the orientation on it so that it is coherent with the orientation of the 0-handle. Note that there is a choice to be made on the orientation of the 2-handle. In particular, after choosing an orientation on the boundary of the core, $\partial(C^{2}$), of a 2-handle, an orientation is induced on the co-core, $C^{2*}$ of the 2-handle. 

\subsection{Handle decomposition}
\label{subsection: Handle decomposition}

\begin{definition}[Compact surface with boundary]
For a compact surface with boundary, $\Sigma^2$, a {\em handle decomposition of $\Sigma^2$} consists of:
\begin{itemize}
    \item[0.] A finite collection of 0-handles $ \{ B^0 \times B^2\}_p$ such that $\mathbf{H}_0^2:= \bigsqcup_p (B^0 \times B^2) \subset \Sigma^2$, where $\partial \mathbf{H}_0^2 = \bigsqcup_p \partial(B^0 \times B^2) ( \cong \bigsqcup_p S^1)$. 
    \item[1.] A finite collection of 1-handles $ \{ B^1 \times B^1 \}_q$ such that $\mathbf{H}_1^2:= \bigsqcup_q (B^1 \times B^1) \subset \Sigma^2$, and non-intersecting attaching maps $\{ \phi^{\pm} \}_q$, where $\phi^{\pm}: \mathbb{B}^{\pm} \rightarrow \partial \mathbf{H}_0^2$. 
    \item[2.] A finite collection of 2-handles $ \{ B^2 \times B^0 \}_r $ such that $\mathbf{H}_2^2:= \bigsqcup_r (B^2 \times B^0) \subset \Sigma^2$, and non-intersecting attaching maps $ \{ \psi^{\pm} \}_r$, where $\psi^{\pm} :\partial B^2 \times B^0 \rightarrow \partial \left( \mathbf{H}_0^2 \bigsqcup \mathbf{H}_1^2/ \{ \phi^{\pm} \}_q \right)$ 
\end{itemize}
Then, $$\frac{\mathbf{H}^2_0 \sqcup \mathbf{H}^2_1 \sqcup \mathbf{H}^2_2}{\{ \phi^{\pm} \}_p \cup \{ \psi^{\pm} \}_q} \cong \Sigma^2.$$
\end{definition}

The reader should recall that from a handle decomposition we have the Euler Characteristic, $$\chi(\Sigma^2) = \sum_{0 \leq s \leq 2} (-1)^s |\mathbf{H}^2_s|.$$
As with the classical Seifert algorithm, we will be concerned only with oriented surfaces.
For our spanning surface construction we require the following specialized definition of a $3$-manifold handle decomposition.

\begin{definition}[Closed oriented $3$-manifolds]
\label{definition: 3-handle decomposition}
For a closed orientable $3$-manifold, $M^3$, a {\em handle decomposition of $M^3$} consists of:
\begin{itemize}
    \item[0.] A unique 0-handle $\{ B^0 \times B^3 \} = \mathbf{H}_0^3 \subset M^3.$ 
    \item[1.] A finite collection of 1-handles $\{ B^1 \times B^2 \}_g $ such that $\mathbf{H}_1^3 := \bigsqcup_g (B^1 \times B^2) \subset M^3$, and non-intersecting attaching maps $\{ \varphi^{\pm} \}_g$ where $\varphi^{\pm}: \mathbb{B}^{\pm} \rightarrow \partial \mathbf{H}_0^3$.
    \item[2.] A finite collection of 2-handles $\{ B^2 \times B^1 \}_g$ such that $\mathbf{H}_2^3 := \bigsqcup_g (B^2 \times B^1) \subset M^3$ and non-intersecting attaching maps $\{ \varrho^{\pm} \}_g $ where $\varrho^{\pm}: \partial B^2 \times B^1 \rightarrow \partial \left( \mathbf{H}_0^3 \bigsqcup \mathbf{H}_1^3 / \{ \varphi^{\pm} \}_g \right)$.
    \item[3.] A unique 3-handle $\{ B^3 \times B^0 \} = \mathbf{H}_3^3 \subset M^3$ with, up to isotopy, a unique attaching map $\varsigma: \partial \mathbf{H}_3^3 \rightarrow \left( \mathbf{H}_0^3 \bigsqcup \mathbf{H}_1^3 \bigsqcup \mathbf{H}_2^3 / \{ \varphi^{\pm} \}_g, \{ \varrho^{\pm} \}_g \right)$
\end{itemize}

Then
\begin{equation}\label{M3eq}
\frac{\mathbf{H}^3_0 \sqcup \mathbf{H}^3_1 \sqcup \mathbf{H}^3_2 \sqcup \mathbf{H}^3_3}{\{ \varphi^{\pm} \}_g, \{ \varrho^{\pm} \}_g, \{ \varsigma \} } \cong M^3
\end{equation}

\end{definition}

For later convenience, we assume the readily achievable technical condition ($\star$): for each pair of points, $(x, y)$, where $x$ is a point in the core of a handle in $\mathbf{H}^3_1$ and $y$ is a point in the co-core of a handle in $\mathbf{H}^3_2$, $ \{x \} \times \partial B^2_u$ and $\partial B^2_v \times \{y\}$ intersect transversely, $1 \leq u,v \leq g$.

\subsection{The Heegaard Graph}
\label{subsection: Heegaard graph}

From Definition \ref{definition: 3-handle decomposition} we can define, $\h \subset \partial \mathbf{H}^3_0$, a natural planar ``fat graph'' that captures all the information of a $3$-manifold's handle decomposition.

\begin{definition}[Heegaard Graph] 
    Given a Heegaard diagram $(F^2, \alpha, \beta)$ of a genus $g$ Heegaard splitting of a 3-manifold, the corresponding \textbf{Heegaard graph}, $\h$, consists of a collection of (fat) vertices $\{ V_i^{\pm} \}_g$ and edges $\{ e_{j,l} \}$, for $1 \leq i, j \leq g$, $1 \leq l < \infty$, called \textbf{$\h$-vertices} and \textbf{$\h$-edges}, respectively. The $\{ V_i^{\pm} \}_g$ represent the $g$ characteristic curves in $\alpha$, and the $\{ e_{j,l} \}$ are such that $ \{ \bigcup_l e_{j,l}: 1 \leq j \leq g \}$ represent the $g$ characteristic curves in $\beta$. For each $i$, $V_i^{\pm} \bigcap e_{j,l} \not= \emptyset$ for at least one $l$. In particular, each nonempty intersection $V_i^{\pm} \bigcap e_{j,l}$ is a numbered point on $\partial V_i^{\pm}$. For each $i$, $V_i^+ \xleftrightarrow{r} V_i^-$, where $r$ is a reflection. 
\end{definition}

$\h$ has an orientation inherited from its Heegaard splitting. Also, $\h \subset S^2 = \partial(B^0 \times B^3)$. We will denote by $\hand_0$, the unique 0-handle of the Heegaard splitting on which we will be working on. That is, $\h \subset \partial \hand_0$. From \eqref{M3eq} above, it can be seen that the fat vertices, $\{ V^\pm_i \}$, and edges, $\{ e_{j,l} \}$, of a Heegaard graph are precisely the $\varphi(\mathbb{B}^\pm_i)$ and $\varrho(\partial(B^2 \times c(B^1)))$, respectively.

\subsection{Submanifolds in normal position}
\label{subsection: normal position}

For the definitions in this section we assume that $\Sigma^2$ is a submanifold in $M^3$.  Morover, we assume that both $\Sigma^2$ and $M^3$ have handle decompositions.

\begin{definition}[Links in normal position]
\label{definition: normal position for links}
    Let $M^3$ be a closed oriented $3$-manifold with a handle decomposition and $L \subset M^3$ be a link. Then $L$ is in {\em weak normal position with respect to the handle decomposition} of $M^3$ if the following two conditions hold:
  \begin{itemize}
        \item[0.)] $L$ is contained in the union of $\hand_0$ and $\mathbf{H}_1^3$.
        \item[1.)] (intersection with $\mathbf{H}_1^3$) For each $1$-handle, $B^1 \times B^2$ of $\mathbf{H}_1^3$, with co-core $C^{1\ast}$, we require that $L \cap C^{1 \ast}$ be a discrete set of points so that we have $L \cap [B^1 \times B^2] = B^1 \times  [L \cap int(C^{1 \ast})]$.
        \end{itemize}
The link is in {\em normal position with respect to the handle decomposition} if it also satisfies the following:
        \begin{itemize}
        \item[2.)] (intersection with $\hand_0$) We require that $L \cap \hand_0 = L \cap \partial \hand_0$-- a collection of pairwise disjoint arcs and circles such that
        \begin{itemize}
            \item[i.)] Each arc has one endpoint contained in the interior of $\varphi(\mathbb{B}^\pm_i)$ and the other contained in the interior of $\varphi(\mathbb{B}^\pm_j)$, and its intersection with $\varphi(\mathbb{B}^\pm_k)$ is empty for $k \not= i,j.$
            \item[ii.)] Each circle component of $L \cap \partial \hand_0$ is away from the $\varphi(\mathbb{B}_i^\pm)$ (fat vertices) of $\h$.
        \end{itemize}
         \end{itemize}
We observe that every link is isotopic to one in weak normal position but not all links contained in $M^3$ are isotopic to a link in normal position since an isotopy that pushes a link so as to be contained in $\partial \hand_0$ may result in double points (crossings)---the collection of arcs and circles are no longer simple or pairwise disjoint.  With this in mind, we say a link is in {\em pseudo-normal position} if it satisfies condition--2.) except in a neighborhood of finitely many points in $\partial \hand_0$ where a crossing occurs.
\end{definition}

\begin{definition}[Surfaces in normal position]
\label{defintion: normal position for surfaces}
    We say the handle decomposition of $\Sigma^2$ is in {\em normal position} with respect to the handle decomposition of $M^3$ if, first, $\Sigma^2$ has boundary, $\partial \Sigma^2 = L$, such that $L$ is a link in normal position, and 
    \begin{itemize}
        \item[0.)] $\Sigma^2$ is contained in the union of $\hand_0$, $\mathbf{H}_1^3$, and $\mathbf{H}_2^3$.
        \item[1.)] (intersection with $\hand_0^3)$ $\Sigma^2 \cap \hand_0^3$ is a collection of pairwise disjoint properly embedded discs. 
        \item[2.)] (intersection with $\mathbf{H}_1^3$) For each $1$-handle, $B^1 \times B^2 \subset \mathbf {H}_1^3$ and its co-core, $C^{1\ast}$, we have that $\pi^\ast : \Sigma^2 \cap (B^1 \times B^2) \rightarrow C^{1\ast}$ is a collection of pairwise disjoint arcs which can be of three possible types: 
        \begin{itemize}
            \item[i.)] Type-1: Properly embedded in $C^{1\ast}$.
            \item[ii.)] Type-2: Embedded, having one endpoint on $\partial C^{1\ast}$ and one endpoint in $int(C^{1\ast})$. 
            \item[iii.)] Type-3: Embedded, having both endpoints in $int(C^{1\ast})$.  (These arcs are also known as {\em ribbon arcs}.)
        \end{itemize} 
        Thus, $\Sigma^2 \cap (B^1 \times B^2) = B^1 \times (\Sigma^2 \cap C^{1\ast})$, is a collection of rectangular discs of Type-1, Type-2, and Type-3, to expand the terminology in the obvious way.
        \item[3.)] (intersection with $\mathbf{H}_2^3$) For a $2$-handle, $B^2 \times B^1 \subset \mathbf{H}^3_2$, and its co-core, $C^{2\ast}$, we have that $\pi^\ast : \Sigma^2 \cap (B^2 \times B^1) \rightarrow C^{2\ast}$ is a finite collection of points in $int(C^{2\ast})$. Thus, we have $\pi^{*-1}(\pi^*(\Sigma^2 \cap B^2 \times B^1)) = \Sigma^2 \cap (B^2 \times B^1)$, a collection of $2$-discs all parallel to $C^{2}$.
    \end{itemize}
\end{definition}

By well known general position arguments, one can position any closed essential surface in an irreducible $3$-manifold so as to have the handle decomposition of the $3$-manifold imposed on the surface.  But, surfaces with boundary present as issue.  It is possible to position them so that their intersections with $\mathbf{H}^3_i, 1 \leq i \leq 2,$ mimic condition $2$ and $3$.  However, since $\partial \Sigma^2 \cap \hand_0$ can be an arbitrary tangle of $k$-stands with additional $S^1$ components, condition $1$ in general cannot be satisfied.  The issue of $\partial \Sigma^2 \cap \hand_0$ being a tangle is the motivation for the definitions in our next subsection.

\subsection{Link projection in Heegaard graph}
\label{subsection: Link projection}

For the unit interval, $I = [0,1]$, a proper embedding, $(I,\partial I) \rightarrow (B^3 ,\partial B^3)$, is a {\em strand}.  The image of the ``$0$'' endpoint is the {\em negative endpoint} and the image of the ``$1$'' endpoint is the {\em positive endpoint}. For an oriented, $S^1$, an embedding, $S^1 \rightarrow int(B^3)$, is a {\em circle}.

\begin{definition}[Tangles in $\hand_0$]
    An $(s,c)$-tangle (or just tangle), $T \subset B^3$, consists of a proper embedding of $s$ strands and an embedding of $c$ circles into a $3$-ball, all of which are pairwise disjoint.  Recalling the graph, $\h \subset \partial \hand_0$, a tangle, $\mathbf{T} \subset \hand_0$, is an \textbf{$\h$-tangle} if the endpoints of the strands are contained in the interior of the $\h$-vertices such that the reflection, $V_i^+ \xleftrightarrow{r} V_i^-$, sends negative (resp., positive) endpoints of strands to positive (resp., negative) endpoints of strands.  
    \end{definition}

For the following definition we consider the projections of an $\h$-tangle onto $\partial \hand_0$.

\begin{definition}
\label{definition: regular projection}
    Let $\mathbf{T}$ be an $\h$-tangle. Let $\mathbf{\pi}_0 : \mathbf{T} \rightarrow \partial \hand_0$ be a projection of the $\h$-tangle onto the $2$-sphere boundary, $\partial \hand_0$.  Then $\mathbf{\pi}_0$ is \textbf{$\h$-regular} if the following hold: 
    \begin{itemize}
        \item (Crossing condition) All crossings (which are double points) of $\pi_0(\mathbf{T})$ are in the regions of $\partial\hand_0 - \h$.
        \item (Vertex condition) For each $\h$-vertex, $V^+_i$, we have $\pi_0(\mathbf{T}) \cap V^+_i$ is a collection of $k$ arcs where $k$ is the number of strand endpoints in int($V^+_i$).  Further, we require that there be similar arcs in $V^-_i$ such that the two collections of strand endpoints respect the reflection map,  $V_i^+ \xleftrightarrow{r} V_i^-$.
        \item (Edge condition) Each edge, $e_{j,l} \subset \h$, intersects $\pi_0(\mathbf{T})$ transversely resulting in a finite collection of points. 
    \end{itemize}
\end{definition}
    We observe that the orientation of $\mathbf{T}$ induces an orientation of the projection, $\pi_0(\mathbf{T})$.

For every link, $L \subset M$, in weak normal position we have the associated $\h$-tangle, $L \cap \hand_0$.  We will take the {\em $\h$-regular projection, $\pi_0(L)$,} to be $\pi_0 (L \cap \hand_0)$.
   
Finally, we observe that we can obtain an oriented link in normal position from $\pi_0(L)$ by performing the classical operation of resolving the crossings of $\pi_0(L)$ in an oriented fashion.  We will denote such a link in normal position obtained from $\pi_0(L)$ by $L_N$.

\subsection{Seifert's algorithm on tangles}
\label{subsection: Link projection}

In this section we describe Seifert's algorithm specialized to $\h$-tangles.  We assume that we are given a $3$-manifold, $M^3$, along with a handle decomposition that results in a Heegaard graph, $\h$, on the boundary of the unique $0$-handle, $\hand_0$.  We begin with the following definition.
\begin{definition}
    An $\h$-vertex is \textbf{balanced} if it contains the same number of positive and negative endpoints of an $\h$-tangle $\mathbf{T}$. An $\h$-tangle, $\mathbf{T}$, is \textbf{balanced} if  each $\h$-vertex in $\h$ is balanced.   We say a link, $L \subset M$, in weak normal position is {\em balanced} if its associated $\h$-tangle, $L\cap \hand_0$, is balanced.
\end{definition}

\begin{proposition}
\label{proposition: balanced}
    Let $L \subset M^3$ be an oriented link such that $L$ is in normal position.  Then $L$ bounds an oriented surface (in normal position), $\Sigma^2 \subset \mathbf{H}_1^3 \cup \hand_0$, if and only if $L$ is balanced.
\end{proposition}

\begin{proof}
If we assume that $L$ bounds an oriented surface, $\Sigma^2$, that is totally contained in $\hand_0$ and the $1$-handles of the decomposition of $M^3$, then it is readily seen that $\Sigma^2$ will intersect the co-cores of $\mathbf{H}^3_1$ in a collection of ribbon arcs (Type-3 arcs).  (There could possibly be circle intersections, but our argument does not need to deal with them.)  Under any orientation of the co-cores, these arcs give a ``pairing'' of positive and negative puncture points.  This pairing can be transcribed to the vertices of the $\h$-graph, showing that the associated $\h$-tangle is balanced.

Proving the other direction, we assume that $L$ is in normal position and that the associated $\h$-tangle is balanced.
For each vertex, $V^+_i , \ 1 \leq i \leq g$, of $\h$ we can then choose ``pairing paths'' in $int(V^+_i) \setminus int(L \cap V^+_i)$ that will connect each positive endpoint of $L \cap V^+_i$ to a negative endpoint of $L \cap V^+_i$.  It can be readily seen that such a choice is always possible but not always unique.  The result is, $\{e(i)^+_1, \cdots , e(i)^+_{k_i}\} \subset int(V^+_i) \setminus int(L \cap V^+_i)$, a collection of embedded pairwise disjoint arcs, each one containing a positive/negative pair of endpoints of $L \cap \partial \hand_0 (= L \cap \hand_0)$.

We can then obtain a corresponding collection of arcs, $\{e(i)^-_1, \cdots , e(i)^-_{k_i}\} \subset V^-_i$, using the reflection map, $V_i^+ \xleftrightarrow{r} V_i^-$.

With these pairing paths in the $\h$-vertices in place, we now observe that the union
$$\mathbf{P}= L \sqcup [\cup_{i,j} (e(i)^+_j \cup e(i)^-_j)] (\subset \partial \hand_0)$$
corresponds to an oriented unlink projection onto a $2$-sphere.  We then choose a pairwise disjoint collection of spanning discs, $\mathbf{D} \subset \hand_0$, having $\mathbf{P}$ as its boundary.

Finally, for each set of pairing segments, $ e(i)^+_j \subset V_i^+ , \  e(i)^-_j \subset V_i^-$, that correspond to each other by a reflection map, we attach a rectangular $2$-disc band in the associated $1$-handle.  Two non-adjacent sides of the band will necessarily be on the link, $L$, as it passes through the one handle, and the other two non-adjacent sides of the band will be glued to the pairing segments. We now have a surface satisfying Definition 5, vacuously satisfying condition 3.), and satisfying Type-3 of condition 2.), intersecting $\h_0$ in discs. 

We observe that these attaching bands are in normal position with respect to the $1$-handles and the resulting surface, $\Sigma^2$, is oriented.
\end{proof}

An {\em extension disc}, $E$, is obtained by taking a disc parallel to the core, $C^2$ of a $2$-handle in $\mathbf{H}^3_2$ and adding an annular collar to it, ``extending'' it into the $1$-handles and $\hand_0$.  We say a link of $k$-components, $L_E \subset M$, is an {\em extension link} if it is in normal position and $L_E$ bounds a collection of pairwise disjoint extension discs, $\{E_1 , \cdots , E_k \}$.  The salient observation to make here is that an extension link bounds an oriented (non-connected when $k>1$) surface---a collection of extension discs---that are in normal position.

\begin{proposition}
\label{proposition: extension balanced}
    Let $L \subset M$ be an oriented link in weak normal position and assume $M$ is an integral homology sphere with a given handle decomposition.  Then there exists an extension link of $k$-components, $L_E$, such that, $\cup_{1 \leq i \leq k} E_i \subset M \setminus L$ and $ L_E \sqcup L \subset M $ and is balanced.  
\end{proposition}

\begin{proof}
To begin, for the moment we ignore our assumption that $M$ is a $\mathbb{Z}$-homology sphere and we consider how to write down a finite representation of its first homology, $H_1(M; \mathbb{Z})$, utilizing its handle decomposition structure of genus $g$.  As previously discussed, we assume an orientation on the cores of the $1$-handles, $C^1_i , \ \text{for} \ 1 \leq i \leq g$.  Then we can readily obtain a set of generators by extending each $C^1_i$ in $\hand_0$ to an oriented loop, $\hat{C}^1_i$.  We remark that the oriented link, $\sqcup_{1 \leq i \leq g} \hat{C}^1_i \subset M$, results in a $g$-strand tangle in $\hand_0$ and it is convenient to choose this link so as to have an $\h$-regular projection that is without any crossings.

Building on previously established geometry, the orientations of $C^1_i$ and $C^{1\ast}_i$ have been chosen so that the resulting algebraic intersection, $\hat{C}^1_i \overset{a}{\cap} C^{1\ast}_i = 1$.  Thus, reversing orientation we have, $-\hat{C}^1_i \overset{a}{\cap} C^{1\ast}_i = -1$.   More generally, $\hat{C}^1_i \overset{a}{\cap} C^{1\ast}_j = \delta_{ij}$.

We denote the homology class, $A_i = [\hat{C}^1_i], \ 1\leq i \leq g$.  Then $\{A_i| \ 1 \leq i \leq g\}$ is a set of generators for $H_1(M; \mathbb{Z})$.
Then, for an oriented knot, $K \subset M$, the homology class, $[K] \in H_1 (M ; \mathbb{Z})$, is just $$[K] = r_1 A_1 + \cdots + r_g A_g,$$
where $r_i = K \overset{a}{\cap} C^{1\ast}_i, \ \text{for} \ 1 \leq i \leq g$.  In particular, for each $2$-handle, let $E_j, \ 1 \leq j \leq g$, be an extension disc.  Then we obtain $g$ relators for our representation of $H_1(M; \mathbb{Z})$, $$R_j = [\partial E_j] = r_{1j} A_1 + \cdots + r_{gj} A_g.$$
Thus, $$H_1(M; \mathbb{Z}) = \{ A_1, \cdots A_g | R_1, \cdots R_g \}.$$

Now, let $L \subset M$ be an oriented link in normal position.  Let $\ell_i$ be $g$ coefficients such that $[{L}] = \sum_{1 \leq i \leq g} \ell_i A_i$.  If we now bring back our assumption that $H_1(M; \mathbb{Z}) =0$, then we know that there is an integral solution, $(x_1, \cdots, x_g) \in \mathbb{Z}^g$, to the equation $$\sum_{1 \leq i \leq g} \ell_i A_i = \sum_{1 \leq j \leq g} x_j R_j.$$

Our extension link is then obtained by letting 

    $$ L_E = \bigcup_g \left(\bigsqcup_{| x_i|} -\sigma(x_i) \partial E_i \right),$$

where the subscript, $|x_j|$, on each $\text{``}\sqcup\text{''}$ is interpreted as ``$|x_j|$ parallel copies''.  By construction, 
\begin{equation}
\label{equation: extension}
[L \sqcup L_E ] = \sum_{1 \leq i \leq g} \ell_i A_i - \sum_{1 \leq j \leq g} x_j R_j = 0.
\end{equation}
$L_E$ will have $k$ components, where $k = \sum |x_j|$. Necessarily, the link $L \sqcup L_E$ bounds an oriented surface contained in $\hand_0$ and the 1-handles, and is therefore balanced by Proposition \ref{proposition: balanced}.
\end{proof}

\subsubsection{Generalized Seifert's algorithm}
\label{subsection: algorithm} Let 
Let $L \subset M$ be an oriented link in a homology sphere with a given handle decomposition.  We assume $L$ is in weak normal position.  Applying Proposition \ref{proposition: extension balanced}, let $L_E$ be an extension link such that the oriented link, $L \sqcup L_E \subset M$.  Assume $L \sqcup L_E$ is in weak normal position.  By construction, it is balanced with respect to the handle decomposition.  We then take an $\h$-regular projection, $\pi_0(L \sqcup L_E)$, and consider the associated oriented, pseudo-normal positioned link.  We observe that at each crossing of $\pi_0(L)$ with $\pi_0(L_E)$, the component of $L_E$ is the under-strand---closer to $\partial \hand_0$.

Now let, $(L \sqcup L_E)_N$, be the oriented link in normal position obtained from $\pi_0(L \sqcup L_E)$ by resolving, in an oriented fashion, the crossing(s) of the projection.  We observe that the components of $L_E$ may no longer be contained as components of $(L \sqcup L_E)_N$.

By construction, $(L \sqcup L_E)_N,$ is balanced.  Applying Proposition~\ref{proposition: balanced}, we construct an oriented surface in normal position, $\Sigma^2$.  Specifically, $\Sigma^2$ will be composed of $0$-handles in $\hand_0$ and $1$-handles in $\mathbf{H}^3_1$.

We reconstruct our original link $L \sqcup L_E$ by adding in a half-twisted band at each previously resolved crossing.  These bands will be attached to the boundary of the $0$-handles of $\Sigma^2$.  Their attachment takes the surface, $\Sigma^2$, which is in normal position, to a surface that is in {\em pseudo-normal position}, $\Sigma^2_p$.   That is, similar to Seifert's original construction, these half-twisted bands are positioned in the $0$-handle of $M$, not in its $1$-handles.

Finally, the boundary of $\Sigma^2_p$ will contain the link $L_E$.  We then extend $\Sigma^2_p$ by capping off each component of $L_E$ with the associated extension disc.  Denoting the resulting surface by $\Sigma^2_E$, we say it is in {\em pseudo-normal position} with respect to the handle decomposition---$1$-handles of $\Sigma^2_E$ comes in two flavors, half-twisted bands in $\hand_0$ and ones in $\mathbf{H}^3_1$.

The above construction yields the following corollary.

\begin{corollary}[Generalized Seifert's algorithm]
    Let $L \subset M$ be an oriented link in weak normal position with $M$ being an integral homology sphere.  Then $L$ is isotopic to a link in pseudo-normal position, $L^\prime$, which bounds an oriented surface in pseudo-normal position, $\Sigma^2 \subset M$.  Moreover, $\Sigma^2$ is algorithmically constructed.
\end{corollary}

\begin{remark}
It is readily observed that our generalized Seifert's algorithm can be applied to any oriented link in an arbitrary closed, oriented $3$-manifold when the link is homologically trivial.  All that is necessary is that we have an integral solution to Equation~(\ref{equation: extension}). 
\end{remark}

\section{Example: The Poincar\'e homology sphere.}
\label{section: P-homology sphere}

In light of the terminology and framework described, we use as an example, the Poincar\'e homology sphere, $P^3$,  in order to demonstrate the utility of working on the planar diagram, $\h$. We refer the reader to \cite{Rolfsen} for a more detailed construction of $P^3$. Figure \ref{PoincareHom1} shows $\h$ for a genus-2 Heegaard diagram, $(F^2, \alpha, \beta)$, of $P^3$. The vertices, $V_1^\pm$ and $V_2^\pm$, colored blue and red respectively, represent the two characteristic curves in $\alpha$, and edges $e_{j,1}$ and $e_{j,2}$ colored orange and black respectively, represent the images of the two characteristic curves in $\beta$.

\begin{figure}[H]
    \centering
    \begin{minipage}{0.45\textwidth}
        \centering
        \includegraphics[width=1.1\textwidth]{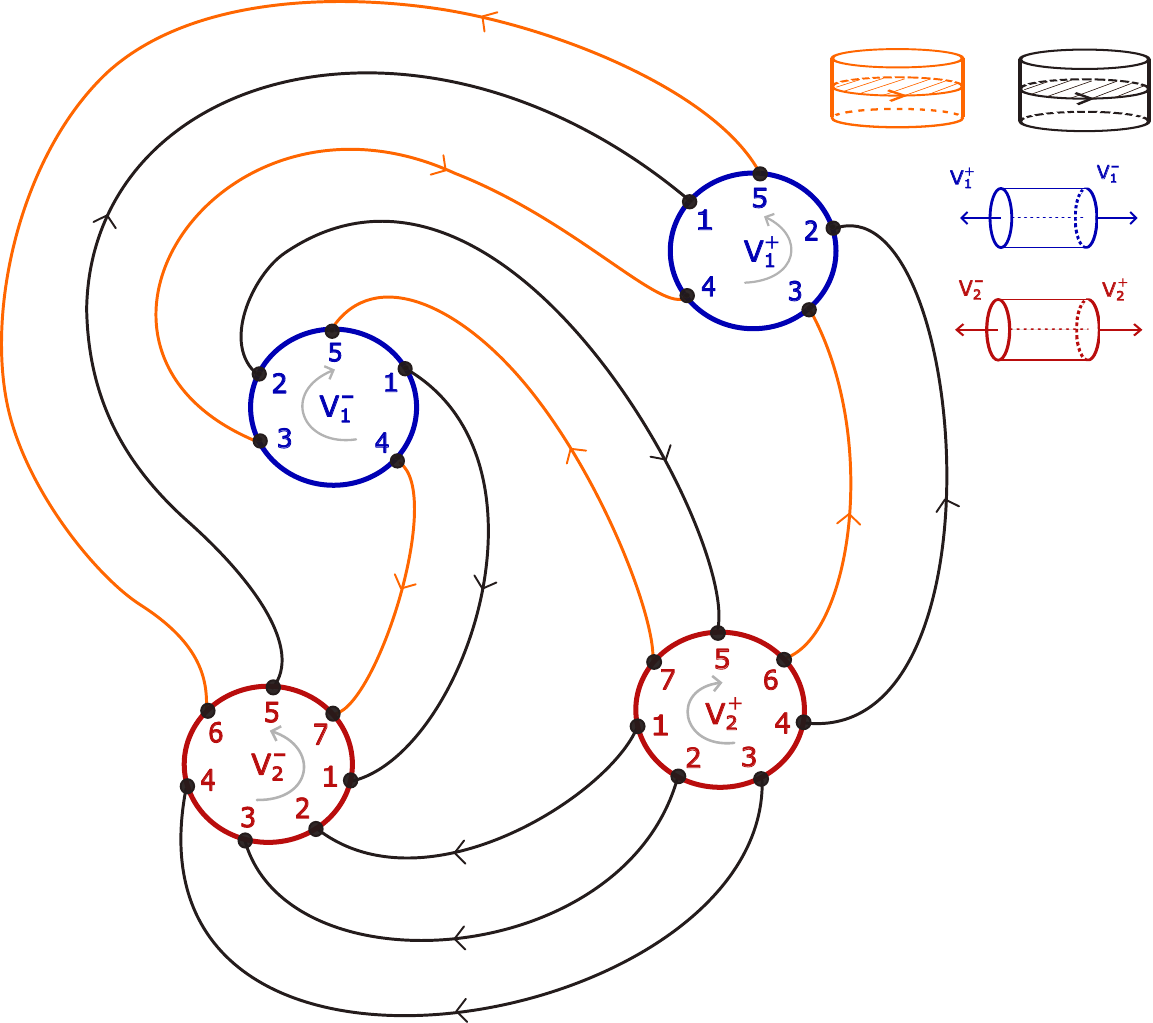} % first figure itself
        \caption{$\h$-graph of $P^3$}
        \label{PoincareHom1}
    \end{minipage}\hfill
    \begin{minipage}{0.45\textwidth}
        \centering
        \includegraphics[width=1.1\textwidth]{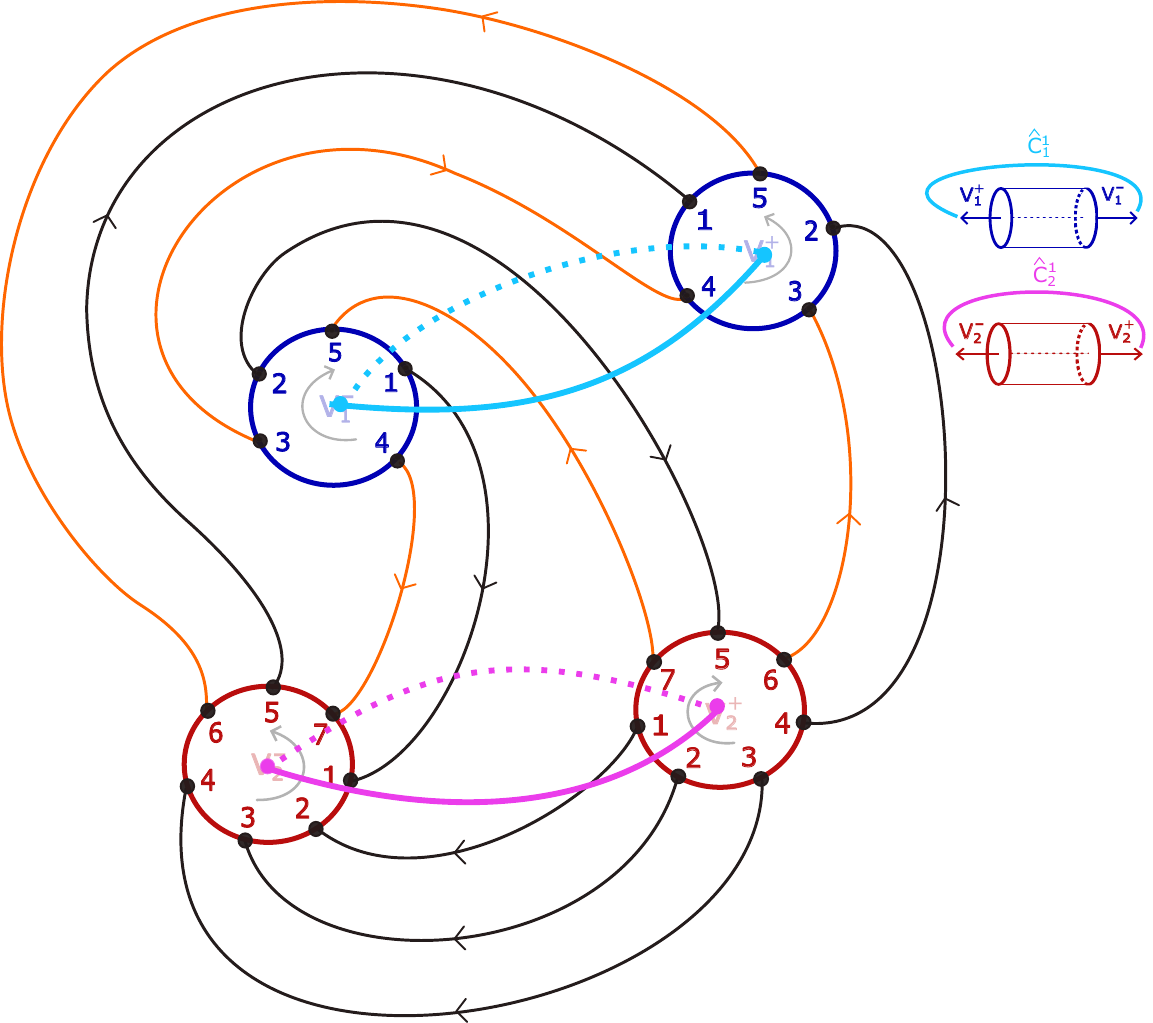} % second figure itself
        \caption{Generators}
        \label{Generators}
    \end{minipage}
\end{figure}

In order to obtain a finite presentation of $H_1(P^3; \mathbb{Z})$ as described in the proof of Proposition \ref{proposition: extension balanced}, we start by extending the cores of the two 1-handles into $\hand_0$ to obtain the generators $A_1 = [\hat{C}^1_1]$ and $A_2 = [\hat{C}^1_2]$ (See Figure \ref{Generators}). Now, we have two relators $R_1 = [\partial E_{1}]$ and $R_2 = [\partial E_{2}]$ associated to the 2-handles colored orange and black, respectively, where $[\partial E_{i}] = r_{1i}A_1 + r_{2i}A_2$ for $i=1,2$. Keeping in mind that $\hat{C}_1^1 \overset{a}{\cap} C^1_1 = 1$ and $\hat{C}_2^1 \overset{a}{\cap} C^1_2 = 1$, we have 
\begin{align*}
    R_1 & = -A_1 + 2A_2 \\
    R_2 & = -2A_1 + 3A_2
\end{align*}

Now, since $P^3$ has trivial first homology, then $A_i = [\hat{C}_{i}] = 0$ for $i=1,2$ and so each $A_i$ bounds a surface. In terms of the relators $R_1$ and $R_2$, we have the following 
\begin{align*}
    A_1 & = x_1R_1 + y_1R_2 = x_1(-A_1 + 2A_2) + y_1(-2A_1 + 3A_2) = (-x_1 - 2y_1)A_1 + (2x_1 + 3y_1)A_2 \\
    A_2 & = x_2R_1 + y_2R_2 = x_2(-A_1 + 2A_2) + y_2(-2A_1 + 3A_2) = (-x_2 - 2y_2)A_1 + (2x_2 + 3y_2)A_2
\end{align*}

which results in the following two systems
  \begin{align*}
    -x_1 - 2y_1 & = 1 & -x_2 - 2y_2 & = 0 \\
    2x_1 + 3y_1 & = 0 & 2x_2 + 3y_2 & = 1
  \end{align*}

Once solved, we get that 
\begin{align*}
    x_1 & = 3 & x_2 & = 2  \\
    y_1 &= -2 & y_2 & = -1
\end{align*}
and so
\begin{align*}
    A_1 &= 3R_1 - 2R_2 \\
    A_2 &= 2R_1 - R_2 
\end{align*}

\begin{figure}[h!]
    \centering
    \includegraphics[width=55mm, height=50mm]{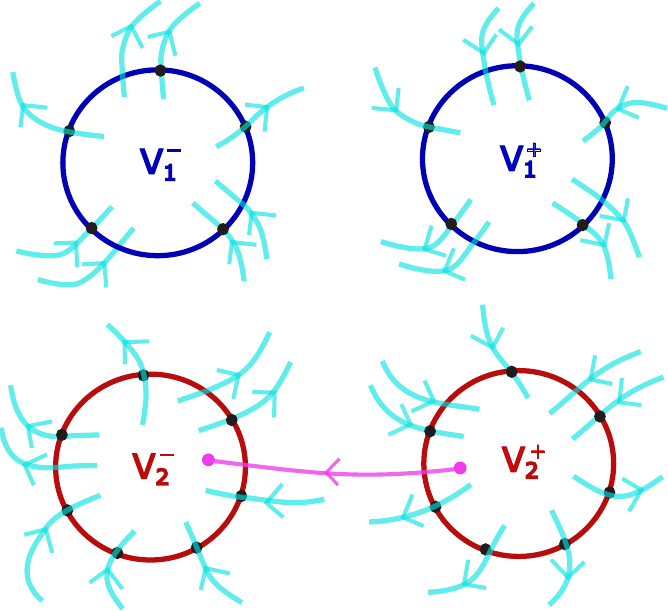}
    \caption{$A_2 \sqcup L_{E_2}$}
    \label{UnbalancedTangleA2}
\end{figure}

Now, $A_2$ is not balanced. The extension link needed to balance $A_2$ is $L_{E_2} = (\bigsqcup_2 -\partial E_1) \bigcup (\partial E_2)$. By construction, $A_2 \sqcup L_{E_2}$ will be balanced (see Figure \ref{UnbalancedTangleA2}) and, by Proposition 9, will bound an oriented surface contained in the 0- and 1- handles of $P^3$. 

Figure \ref{Poincare1_push} shows $L_{E_2}$ projected onto $\h \setminus V_i^\pm$. In order to identify the surface that $A_2 \sqcup L_{E_2}$ bounds, we add in pairing paths as shown in Figure \ref{Poincare1_bands}. 

\begin{figure}[ht]
    \centering
    \begin{minipage}{0.45\textwidth}
        \centering
        \includegraphics[width=0.8\textwidth]{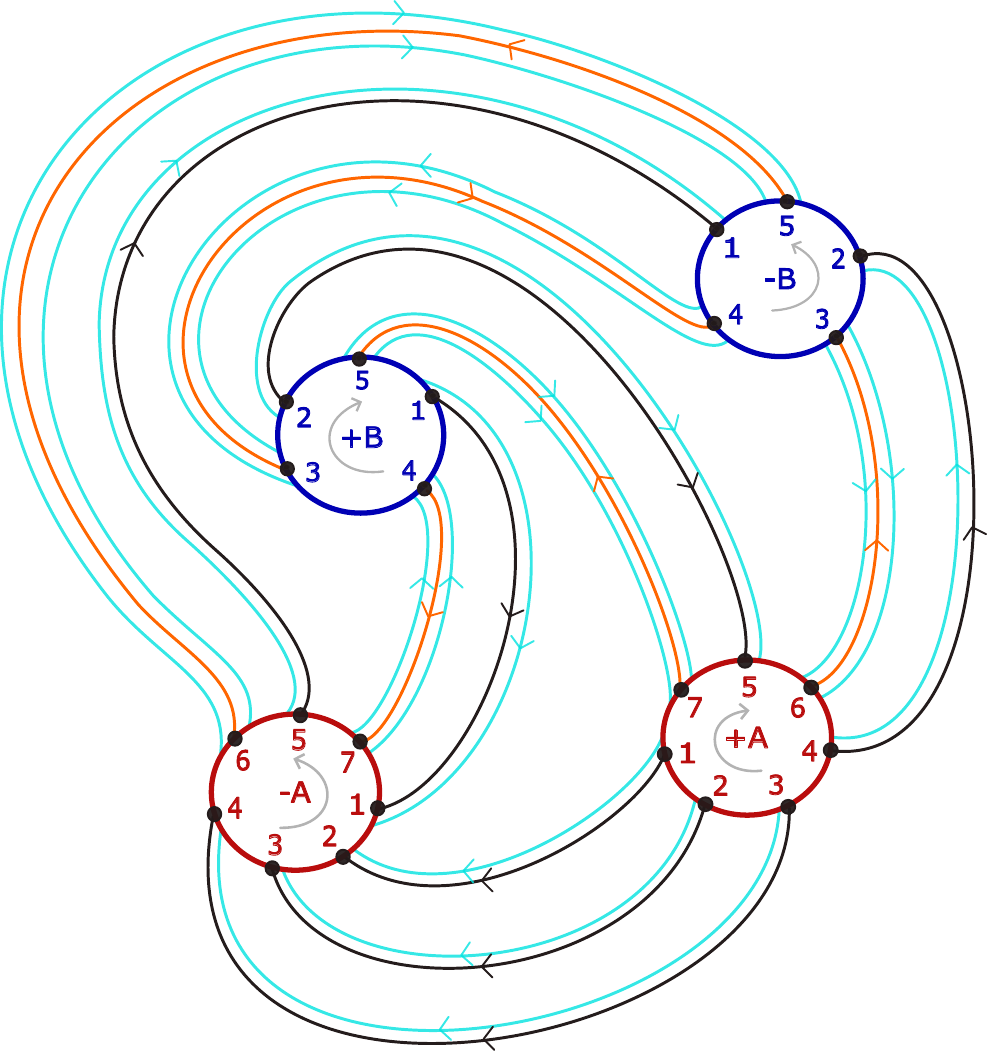} % first figure itself
        \caption{$L_{E_2} \subset \h \setminus V^{\pm}_i $}
        \label{Poincare1_push}
    \end{minipage}\hfill
    \begin{minipage}{0.45\textwidth}
        \centering
        \includegraphics[width=0.8\textwidth]{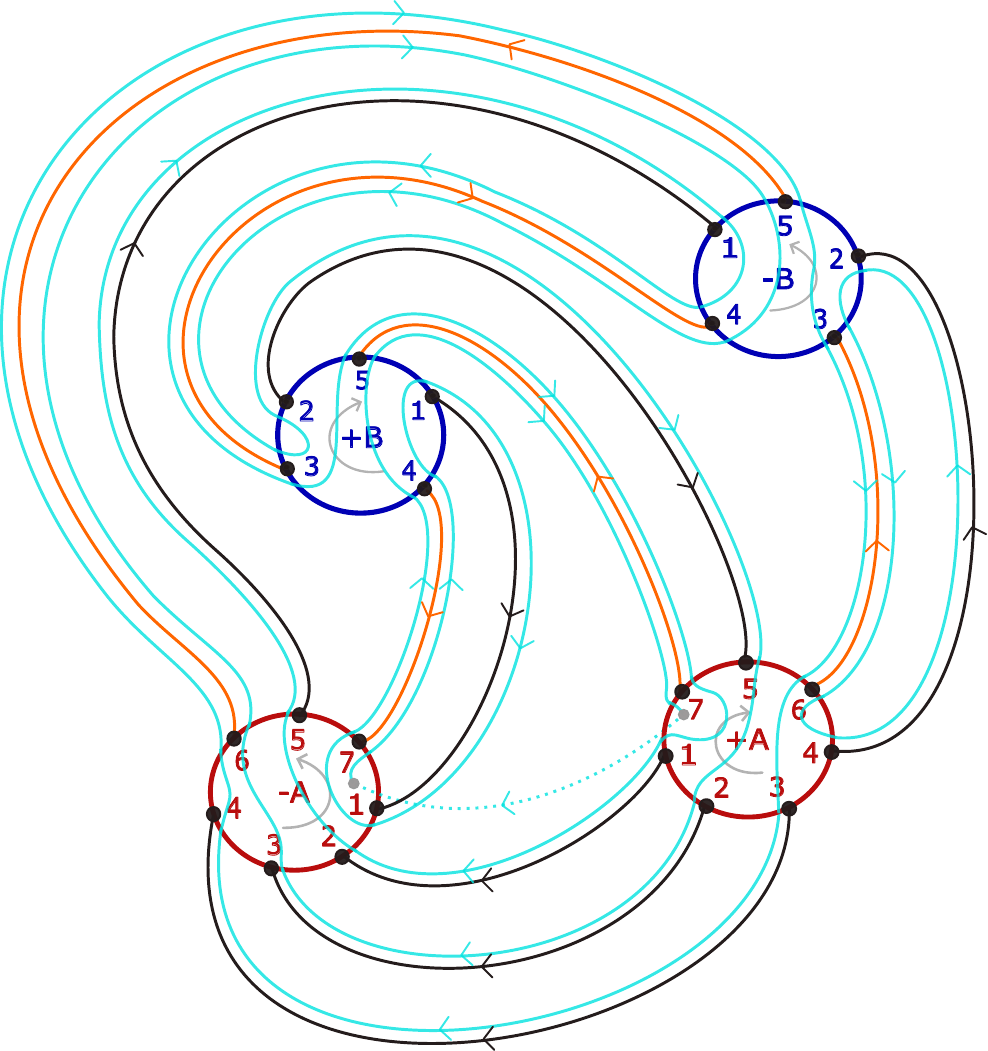} % second figure itself
        \caption{Pairing paths for $L_{E_2}$}
        \label{Poincare1_bands}
    \end{minipage}
\end{figure}

After resolving one crossing, shown in Figure \ref{Poincare1_circles}, the union of $A_2 \sqcup L_{E_2}$ along with the pairing paths will correspond to an oriented unlink projection with spanning disks shown in Figure \ref{Seifert_circles}. A rectangular 2-disc band is attached, in the associated 1-handle, to each corresponding pair of pairing paths. The resulting surface will be oriented, in normal position, and contained in the 0- and 1- handles of $P^3$. Lastly, since each extension disc used in the construction is contained in $P^3 \setminus A_2$, each boundary component of $L_{E_2}$ can be capped off by an extension disc. The result will be an oriented surface, $\Sigma_{E_2}^2$, in normal position having $A_2$ as its boundary. 

\begin{figure}[ht]
    \centering
    \begin{minipage}{0.45\textwidth}
        \centering
        \includegraphics[width=0.8\textwidth]{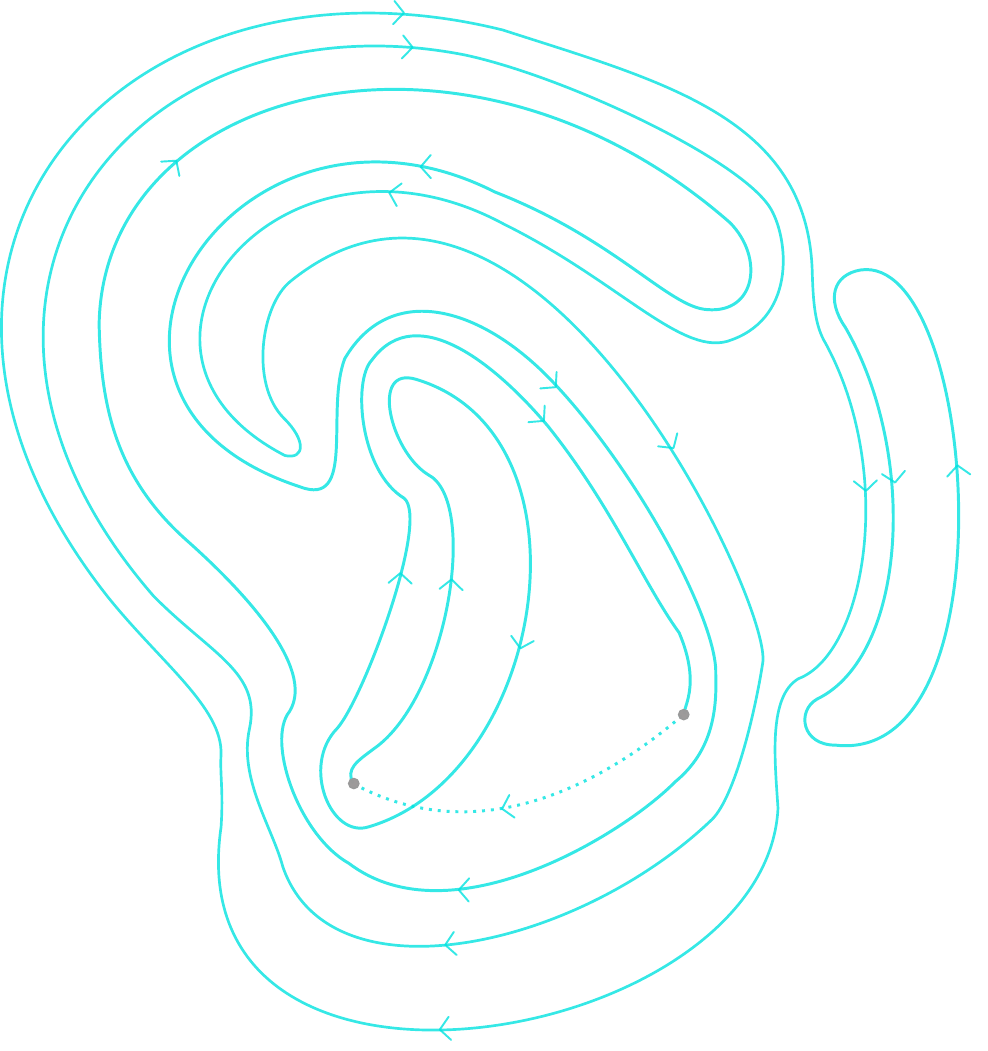} % first figure itself
        \caption{Seifert Circles for $A_2$}
        \label{Poincare1_circles}
    \end{minipage}\hfill
    \begin{minipage}{0.45\textwidth}
        \centering
        \includegraphics[width=0.8\textwidth]{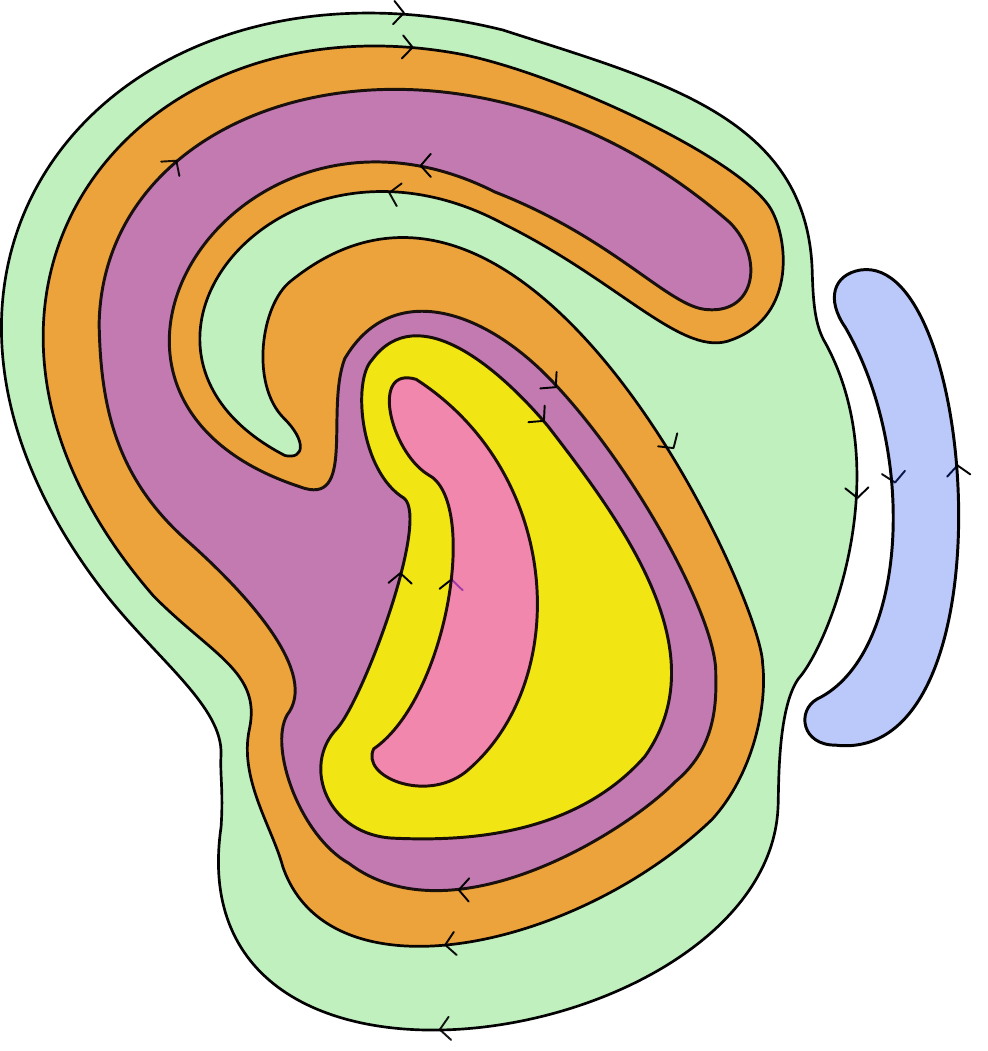} % second figure itself
        \caption{Disks in $\hand_0$ for $A_2$}
        \label{Seifert_circles}
    \end{minipage}
\end{figure}

The Euler characteristic could be calculated by computing the following: \\
$ \chi(\Sigma_{E_2}^2) = \sum_{i=0}^2 (-1)^i (\# \textrm{ of } i-\textrm{handles}) = 6 - 10 + 3 = -1 .$ \\
Since $\Sigma_{E_2}^2$ has one boundary component, namely $A_2$, this surface is one homeomorphic to a once-punctured torus.

We repeat the same procedure in order to identify the surface that $A_1$ bounds. Since $A_1$ is not balanced, the extension disc needed to balance it is $L_{E_1} = (\bigsqcup_3 -\partial E_1) \bigcup (\bigsqcup_2 \partial E_2)$. Therefore, $A_1 \cup L_{E_1}$ will be balanced and the projection of $L_{E_1}$ onto $\h \setminus V_i^\pm$ can be seen in Figure \ref{Poincare2_push}. We add pairing paths as shown in Figure \ref{Poincare2_bands}. 

\begin{figure}[ht]
    \centering
    \begin{minipage}{0.50\textwidth}
        \centering
        \includegraphics[width=0.8\textwidth]{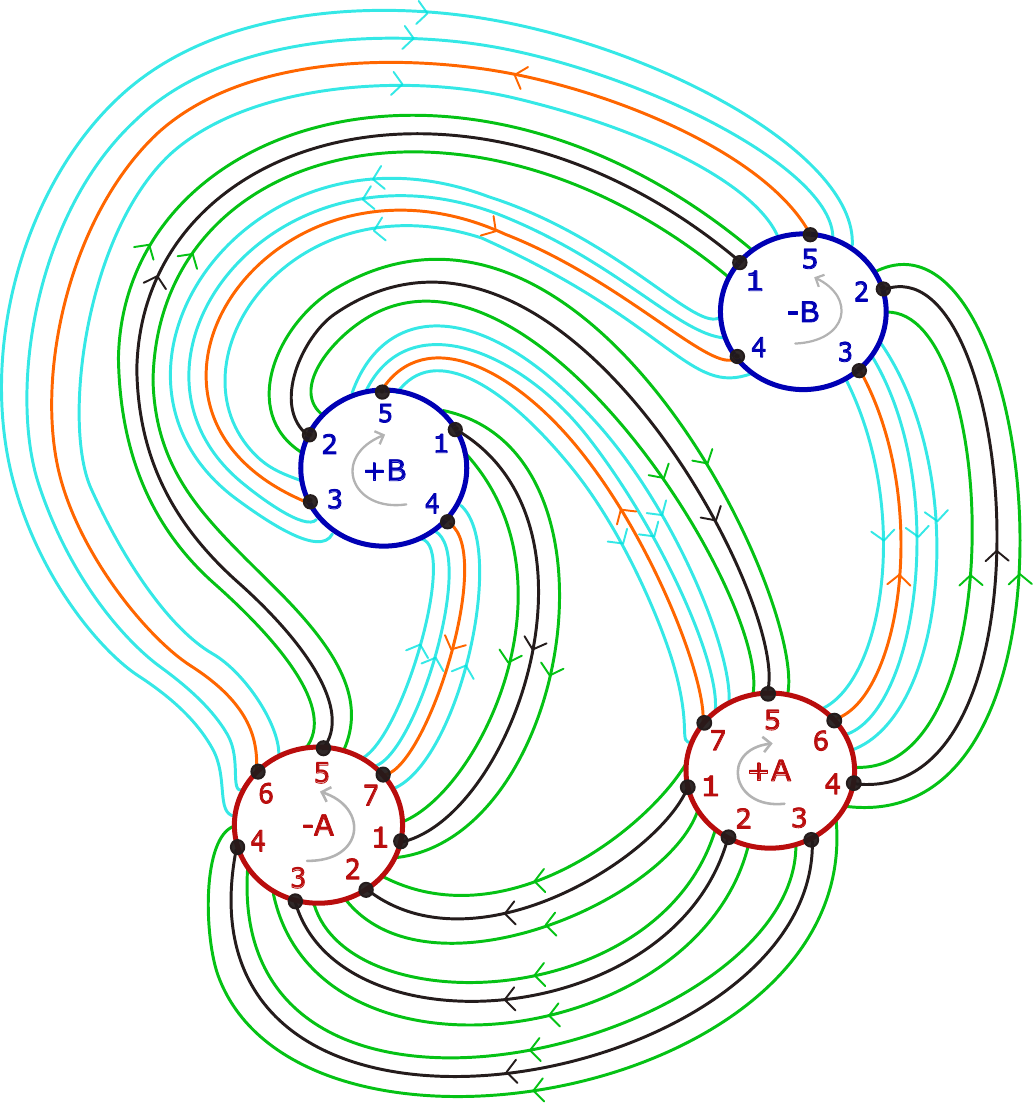} % first figure itself
        \caption{$L_{E_1} \subset \h \setminus V^{\pm}_i $}
        \label{Poincare2_push}
    \end{minipage}\hfill
    \begin{minipage}{0.50\textwidth}
        \centering
        \includegraphics[width=0.8\textwidth]{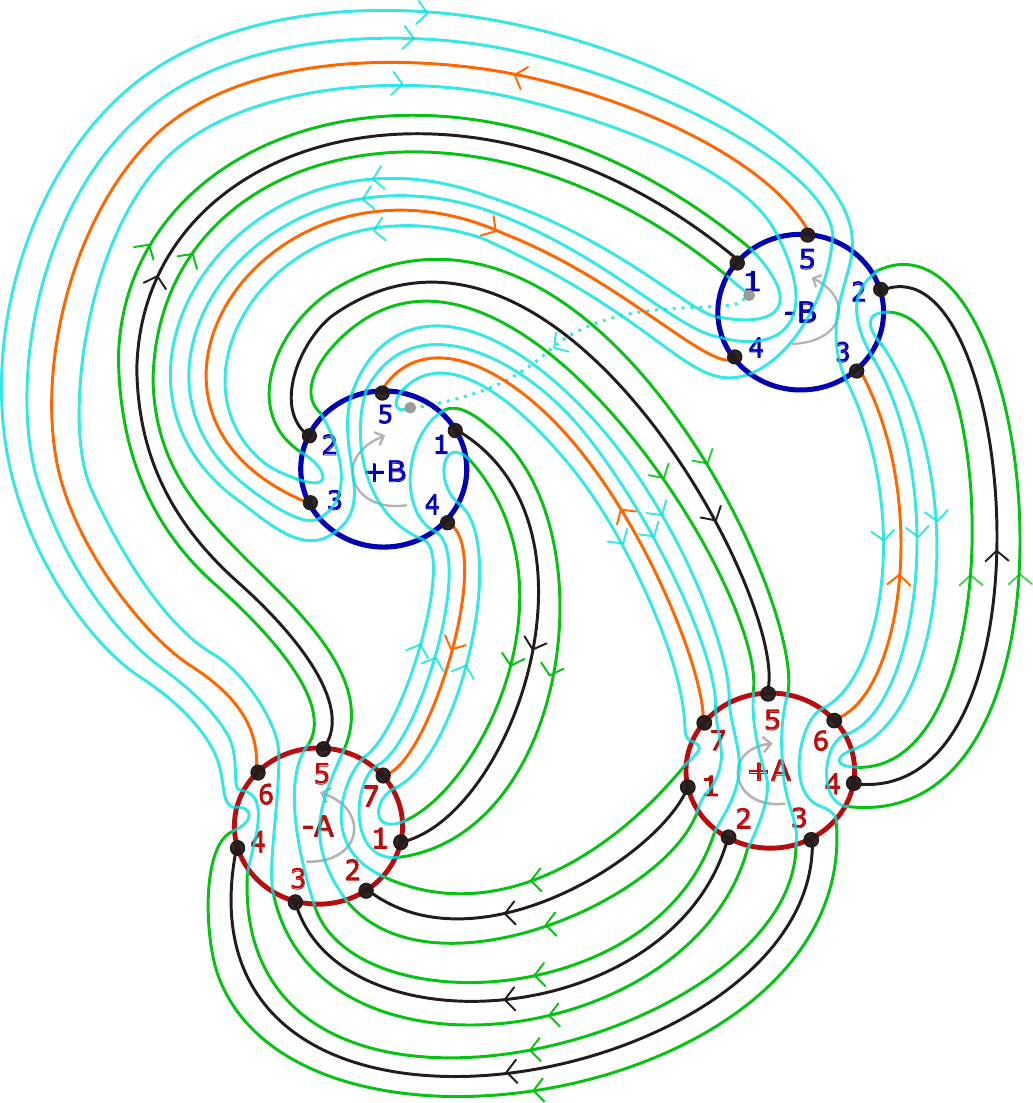} % second figure itself
        \caption{Pairing paths for $A_1$}
        \label{Poincare2_bands}
    \end{minipage}
\end{figure}

After resolving multiple crossings as seen in Figure \ref{Poincare2_crossings}, the union of $A_1 \sqcup L_{E_1}$ along with the pairing paths will correspond to an oriented unlink projection with spanning disks shown in Figure \ref{Poincare2_colored}. After attaching the corresponding rectangular 2-disc bands, the resulting surface will be oriented, in normal position, and contained in the 0- and 1- handles of $P^3$. After capping off each boundary component of $L_{E_1}$ with an extension disc, the resulting surface, $\Sigma_{E_1}^2$, will be in normal position having $A_1$ as its boundary. The surface will have Euler characteristic $\chi(\Sigma_{E_1}^2) = 11 - 23 + 5 = -7 .$

\begin{figure}[h!]
    \centering
    \begin{minipage}{0.45\textwidth}
        \centering
        \includegraphics[width=0.8\textwidth]{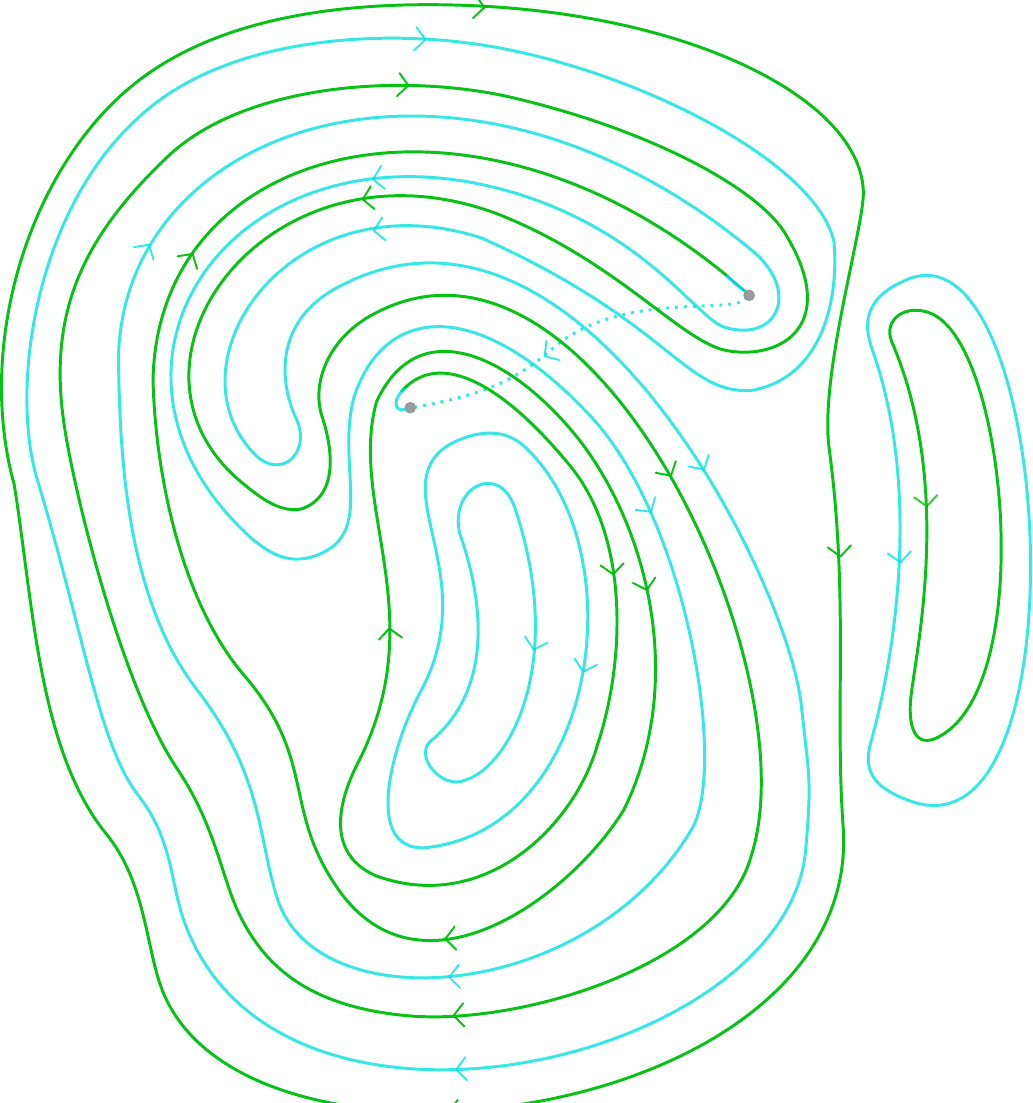} % first figure itself
        \caption{Seifert Circles for $A_1$}
        \label{Poincare2_crossings}
    \end{minipage}\hfill
    \begin{minipage}{0.45\textwidth}
        \centering
        \includegraphics[width=0.8\textwidth]{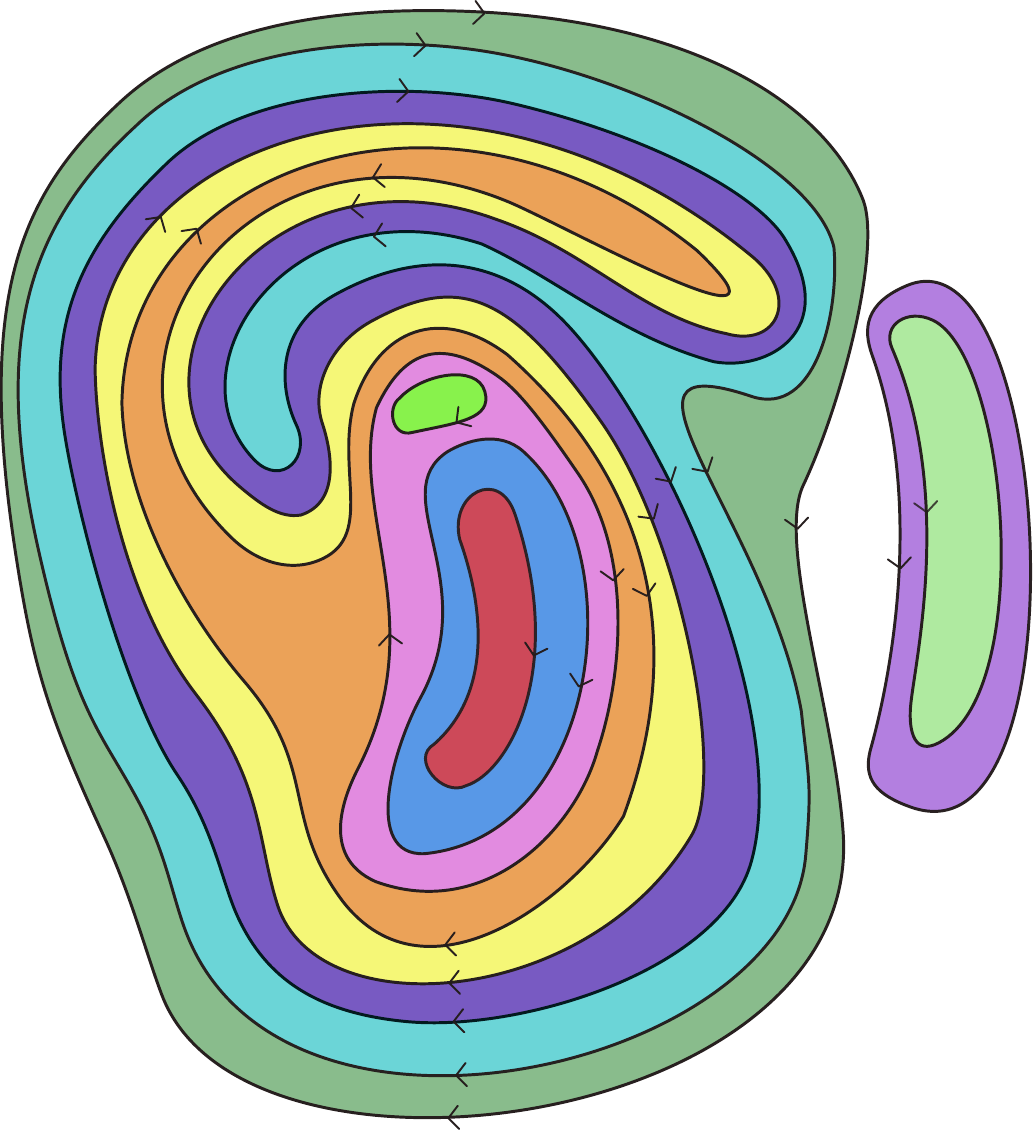} % second figure itself
        \caption{Discs in $\hand_0$ for $A_1$}
        \label{Poincare2_colored}
    \end{minipage}
\end{figure}

\section{Constructing a Heegaard graph from framed links.}
\label{section: Constructing Heegaard graph}

The Heegaard graph similar to the one in \S\ref{section: P-homology sphere} for the Poincar\'e Homology Sphere is the starting point for performing our generalization of Seifert's algorithm. That Heegaard graph comes from D. Rolfsen's ``one-time'' procedure that starts with $+1$ Dehn surgery on the right-handed trefoil knot and, through a sequence of geometric manipulations, decomposes the $3$-manifold into the corresponding handlebody \cite{Rolfsen}.  It is true that we have a ready supply of integral homology spheres coming from $\pm 1$ Dehn surgery on links in $S^3$.  But, to advance the project of studying knots and links in an arbitrary $3$-manifold using link projections in the Heegaard graph, we will need to systematize Rolfsen's original one-time procedure.  To this end, we now describe how a plat presentation of a specified framed link is readily transformed into a Heegaard graph.  

\begin{figure}[h!]

    \centering
    \begin{minipage}{0.8\textwidth}
        \centering
    \includegraphics[width=1.2\textwidth]{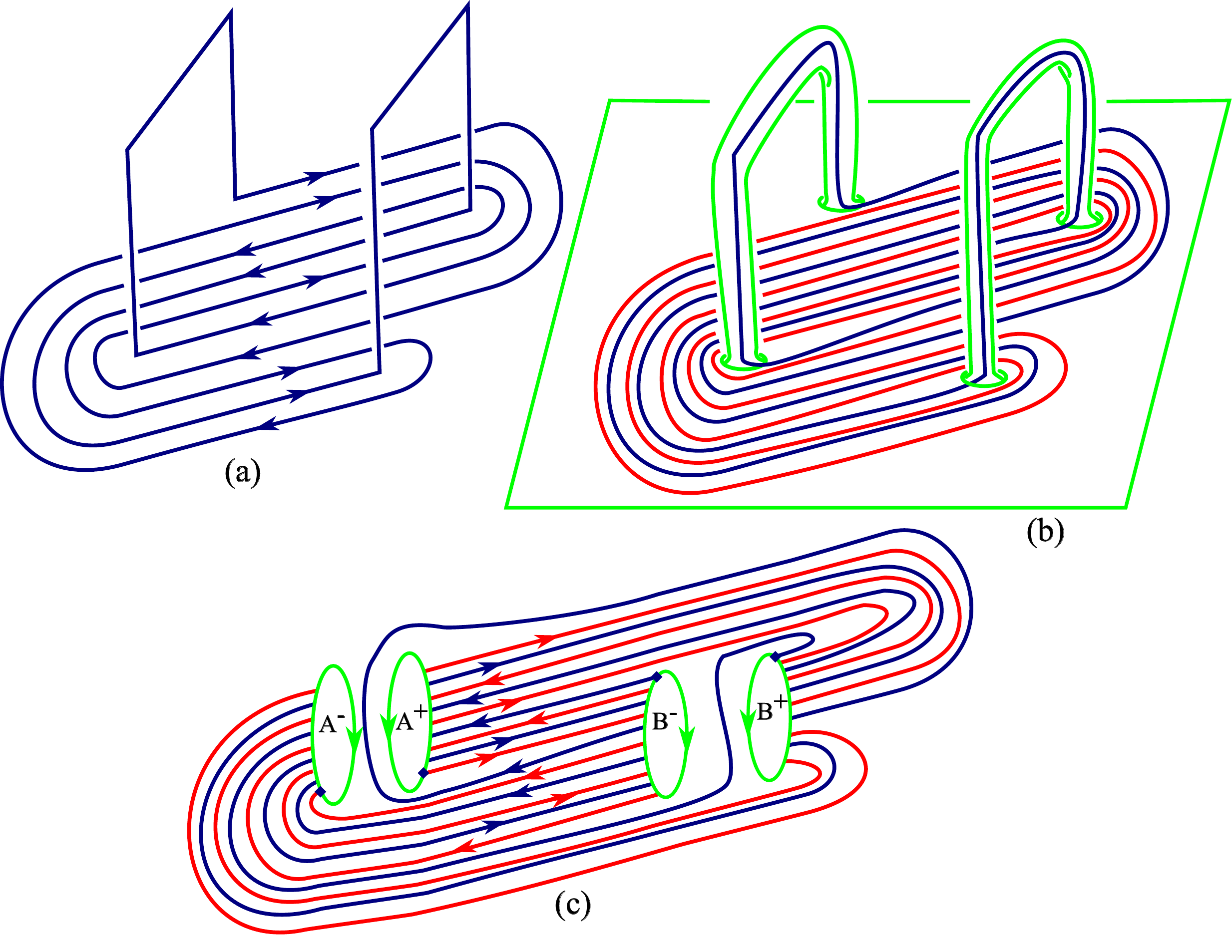}
        \caption{In (a) we depict a plat presentation with $2$ arcs in a plane which are connected above by $2$ bridges.  In (b) we depict the associated genus $2$ splitting surface with $2$ curves (red and blue) that are where the two $2$-handles are attached.  In (c) we depict the associated Heegaard graph.}
        \label{figure: Plat presentation}
        
    \end{minipage}
    
\end{figure}

Our procedure is inspired by the initial figure in \cite{Hatcher-Thurston} which we replicate in Figure \ref{figure: Plat presentation}(a).  Specifically, every link $L$ in $S^3$ has a {\em flat plat presentation}---a decomposition of $L$ into $n$ pairwise disjoint simple arcs in a sphere (or plane) with $n$ (unknotted) bridges positioned ``above'' the sphere connecting the planar arcs together to form the link.  Figure \ref{figure: Plat presentation}(a) illustrates a flat plat presentation for the figure eight knot.  There are $2$ simple arcs in the sphere which are connected by $2$ bridges above the sphere.  It is convenient to depict the $2$ bridges in a rectangular fashion---two vertical edges and one horizontal edge.  This confirms that the bridges are unknotted with respect to the sphere by giving readily identifiable unknotting discs for the bridges---the embedded discs that would be swiped-out by an isotopy that takes the bridges to arcs embedded in the sphere. 

Our procedure for producing a Heegaard graph of the homology sphere produced by $\pm 1$ Dehn surgery on the components of $L$ first requires that we take a flat plat presentation of $L$ and construct a solid handle body of genus $n$, $H_n$.  We will denote the boundary of $H_n$ by $\Sigma_n$.  

Once we have specified $H_n$ we need to go in two different directions.  The first direction is identifying $n$ {\em characteristic curves} in $\Sigma_n$ where $2$-handles will be attached to $H_n$.  The second direction is giving a decomposition of $H_n$ into a single $0$-handle with $n$ $1$-handles attached.

For the first direction, in order to obtain the correct framing for the $n$ characteristic curves, it is convenient to require that the writhe of each component of $L$ correspond to the Dehn surgery that is being applied.  For example, in Figure \ref{figure: Plat presentation}(a) the writhe of the flat plat projection of the figure eight is readily calculated as $+1$.  This will position us to produce a Heegaard graph of a homology sphere that is the result of $+1$ Dehn surgery on the figure eight.  That such a flat plat presentation always exists is a straight forward argument that we leave to the reader.

\noindent
{\bf Identifying $H_n$ and $\Sigma_n$.}  We attach to the sphere $n$ annular tubes, one for each bridge of the plat presentation---one can think of an annular tube as being the boundary of a regular neighborhood of a bridge.  The resulting genus $n$ surface is $\Sigma_n$.  Due to the unknotted nature of the $n$ bridges, it is immediate that $\Sigma_n$ is the boundary of a solid handle body positioned ``above'' the sphere.  Figure \ref{figure: Plat presentation}(b) illustrates the associated genus $2$ surface, $\Sigma_2$.   Our viewpoint is one where we are positioned inside $H_2$.

\noindent
{\bf Identifying characteristic curves in $\Sigma_n$.} Our choice of $n$ characteristic curves on $\Sigma_n$---all of which are non-separating curves---starts with the components of $L$.  $\Sigma_n$ has the property that the components of $L$ are non-separating simple closed curves in $\Sigma_n$.  We observe that the number of components, $|L| \leq n$.  If $|L|=n$ then $L \subset \Sigma_n$ is the required collection of characteristic curves.

If $|L| < n$ then there will be at least $k = n - |L|$ planar arcs, $\{ a_1, \cdots a_k  \}$ of the initial flat plat presentation of $L$ that satisfy the following property $\star \star$: {\sl the two endpoints of $a_i$ are not attached to the same bridge}.  For each such planar arc, we consider the boundary of a regular neighborhood, $\{c_1, \cdots, c_k \}$, all of which are in the sphere that is utilized in our flat plat presentation.  It is readily seen that, except when $n = 2$, the $c_i {\rm 's}$ are in distinct, isotopic classes.  As such, each $c_i$ is positioned away from where we attach the $n$ annular tubes used in the formation of $\Sigma_n$.  By property $\star \star$, each resulting $c_i$ in $\Sigma_n$ will be a non-separating curve.  Our collection of characteristic curves is then all the components of $L$ in $\Sigma_n$ plus a choice of some number of the $c_i {\rm 's}$ so as to give a complete count of $n$ non-separating curves.

Referring the reader to Figure \ref{figure: Plat presentation}(b), the red curve is obtained by looking back at Figure \ref{figure: Plat presentation}(a) and taking the boundary of a regular neighborhood of the planar arc that has the two endpoints, $NE$ and $SW$ (using compass referencing).  Then the two characteristic curves consist of this red curve and the knot that is in blue.

{\bf Decomposition of $H_n$ into a $0$-handle with $n$ $1$-handles.}  We will want to see this handle body as being decomposed into a single $0$-handle with $n$ $1$-handles attached.  Our viewpoint will now be one where we are positioned in the $0$-handle.  To specify the $1$-handles we need only specify the boundary of $n$ co-cores.  These will come from the $n$ discs which characterize the $n$ rectangular bridges as being unknotted.

Using the
two obvious unknotting discs for the two rectangular bridges of Figure \ref{figure: Plat presentation}(b), the reader should now be able to interpret the handle structure of this illustration.  Away from the two annular tubes, one is inside the unique $0$-handle.  When one travels underneath either tube, one is traveling through a $1$-handle.

{\bf Producing the Heegaard graph.}  With the $0$-handle and $1$-handle structure of $H_n$ defined, we are now in a position to produce the associated Heegaard graph.  This is done by splitting $H_n$ along the $n$ $1$-handle co-cores to produce a planar graph.  The vertices of this graph will be the pairs of discs associated with the co-cores.  The edges of this graph correspond to splitting the characteristic curves along their intersections with the $1$-handle co-cores.

Figure \ref{figure: Plat presentation}(c) depicts this splitting procedure coming from our flat plat presentation of the figure eight.  It is a straight forward task to repeat the homological calculation of \S\ref{section: P-homology sphere} and see that the $+1$ writhe framing of Figure \ref{figure: Plat presentation}(a) yields an integral homology sphere.

\section{Uses of a new tool}
\label{section: questions}

{\bf Generalized Alexander Polynomial}--One obvious direction for a generalized Seifert surface algorithm is to carry out each feature of Seifert's original program \cite{Seifert} except now in the setting of an arbitrary homology sphere.  Specifically, coming from the $\pi_0$-regular projection of an oriented link in an $\h$-graph, use the associated Seifert surface in pseudo-normal position to give an effective procedure for calculating an associated ``Seifert matrix''.  Such a (generalized) Seifert matrix will yield an associated ``Alexander polynomial''.
This in fact is the current direction of the first author's work.  The goal is to follow the road map that was laid out in \cite{Seifert} for initially determining the scope of the behavior of such link polynomials.

{\bf Minimal genus spanning surfaces}---A second obvious line of inquiry comes from the algebraic calculation in the proof of Proposition \ref{proposition: extension balanced}.  Specifically, there are ``non-trivial solutions'' to equation (\ref{equation: extension}) coming from the fact that one can add in cancelling extension discs---extension discs of opposite orientation.  Such an addition of extension discs would contribute parallel components to $L_E$ that are oppositely oriented.  However, this adds to the possible choices of the pairing paths used in the proof of Proposition \ref{proposition: balanced}.  So how does the genus of the Seifert surface behave when cancelling extension discs are thrown in?  A reasonable conjecture is that \emph{minimal genus is achieved when $|L_E|$ is minimal}. 

%\newpage

\end{document}